\documentclass{article}
\usepackage{arxiv}

\usepackage[utf8]{inputenc} 
\usepackage[T1]{fontenc}    
\usepackage{graphicx,amsthm,amsmath,amsfonts,amssymb,bbm,dsfont,xcolor,bm,dirtytalk,float,algpseudocode,subcaption, url}

\newcommand{\bomega}{\bm{\omega}}

\newcommand{\bbeta}{\bm{\beta}}
\newcommand{\bxi}{\bm{\xi}}

\newcommand{\bphi}{\bm{\phi}}
\newcommand{\bpsi}{\bm{\psi}}

\newcommand{\rmd}{{\rm{d}}}
\newcommand{\rme}{{\rm{e}}}

\newcommand{\prox}{{\rm prox}}

\newcommand{\by}{\mathbf{y}}
\newcommand{\bz}{\mathbf{z}}

\newtheorem{theorem}{Theorem}
\newtheorem{lemma}{Lemma}
\newtheorem{proposition}{Proposition}
\newtheorem{corollary}{Corollary}
\newtheorem{definition}{Definition}
\title{Proportional asymptotics of piecewise exponential proportional hazards models}

\author{
 Emanuele Massa \\
  Physics of Machine Learning and Complex Systems\\
  Radboud Universiteit\\
  Nijmegen, The Netherlands\\
  \texttt{emanuele.massa@donders.ru.nl} 
}

\begin{document}

\maketitle

\begin{abstract}
We study the flexible piecewise exponential model in a high dimensional setting where the number of covariates $p$ grows proportionally to the number of observations $n$ and under the hypothesis of random uncorrelated Gaussian designs. We prove rigorously that the optimal ridge penalized log-likelihood of the model converges in probability to the saddle point of a surrogate objective function.  The technique of proof is the Convex Gaussian Min-Max theorem of Thrampoulidis, Oymak and Hassibi.
An important consequence of this result, is that we can study the impact of the ridge regularization on the estimates of the parameter of the model and the prediction error as a function of the ratio $p/n > 0$. 
Furthermore, these results represent a first step toward rigorously proving the (conjectured) correctness of several results obtained with the heuristic replica method for the Cox semi-parametric model. 
\end{abstract}

\section{INTRODUCTION}
\label{sec:1}
In medicine and healthcare, survival analysis is extensively used to assess the impact of medical interventions, predict patient outcomes, and estimate disease prognosis \cite{Klein_2005,Kalbfleisch_2011,Ramjith_2022,Cole_2010}.  
Survival data are often analysed in the framework of the Cox model \cite{Cox_1972}, which has several desirable properties in the classical regime where the number of covariates ($p$) is assumed to be small compared to the number of subjects ($n$) available in a study \cite{Gill_82, Tsiatis_1981}. However, in modern applications, $p$ is often comparable or larger than $n$. In this regime the maximum partial likelihood estimator might not even exist, and when it does is frequently biased and exhibits a large variance \cite{Coolen_17, Massa_2024}, making it an unreliable tool for statistical inference and prediction. Such undesirable properties are unavoidable and have been known since the 60' in the statistical community, starting from the works of Kolmogorov and Huber \cite{Wainwright_2019, El_Karoui_13}. 

In the last two decades a vast amount of theory on sparse regression in high dimensional generalized linear models \cite{Van_de_geer_2008}, and the Cox model \cite{Bradic_2011, Shengchun_2014}, has been established. This line of research builds on the assumption that the underlying \say{true} model is sparse and guarantees consistency of estimators obtained by a penalized maximum likelihood approach when : i) the regularizer is cleverly chosen, i.e. lasso and variants of the latter \cite{Tishbirani_1996},  ii) the regularization strength is appropriately tuned \cite{Buhlmann_2011} and iii) the number of components that are actually correlated with the outcome is sufficiently small, i.e. the true model is sufficiently sparse. The theoretical machinery of this line of work, i.e. the theory of empirical processes \cite{Wellner_2013, Van_der_Vaart_2000, Van_de_Geer_2000}, imposes only minor conditions over the distribution of the covariates and the results are often expressed as bounds on performance metrics, with precise probabilistic guarantees \cite{Van_de_Geer_2000, Wainwright_2019, Lederer_2021, Buhlmann_2011, Shengchun_2014}.  

Recently there has been growing theoretical interest in the proportional asymptotic regime where the number of covariates $p=p_n=\zeta n$ grows linearly with the number of observations $n\rightarrow \infty$. If a finite fraction of covariates is associated with the outcome, then even a properly tuned Lasso penalization cannot restore consistency (for all the parameters simultaneously). Hence, the standard theoretical analysis via the Lasso theory is not readily applicable. However, at the price of making more stringent assumptions on the data generating model, an alternative \say{average case} analysis can be carried out.  This line of work, which is sometimes referred to as \say{exact} asymptotics \cite{Celentano_2023, Hastie_2019,Dobriban_2015, Miolane_21}, builds on the standard and mathematically convenient (albeit often ideal) assumption of \say{Gaussian} covariates, i.e. covariates that follow a standard multivariate normal distribution. Heuristic theoretical results of this sort date back to \cite{Gardner_1988}, and has been made rigorous for Generalized Linear Models in the last decade, using different techniques such as Leave One Out method \cite{El_Karoui_18,Sur_19}, Approximate Message Passing \cite{Donoho_2013} and the Convex Gaussian Min max Theorem (CGMT)\cite{Thrampoulidis_2015,Thrampoulidis_2018,Loureiro_2022}.
The \say{average case} approach enables to precisely compute the average of various metrics of interest as a function of the \say{true} parameters, once a conditional model for the response given the covariates is specified. In practice the data generating model is unknown, and it is hard to estimate since the estimators for the parameters of the model are inconsistent. Hence, these results are not directly applicable to real data. However, such a \say{sharp} characterization of prediction and goodness of fit metrics allows comparing different methodologies (e.g. different regularization) in several hypothetical settings.

In this manuscript, we study the asymptotic behaviour of the ridge penalized maximum likelihood estimator for the piece-wise exponential proportional hazards model via a rigorous formulation of the typical case approach mentioned above in the setting where the predictors are uncorrelated, for analytical simplicity. 
The piece-wise exponential proportional hazards model is a parametric proportional hazards model where the base hazard rate function is postulated to be piece-wise constant, throughout a finite number ($\ell$) of \emph{user defined} intervals which does not grow with the sample size $n$ \cite{Friedman_82,Kalbfleisch_73}.
By applying the CGMT, we prove rigorously that the optimal ridge penalized log-likelihood of the model converges in probability to the saddle point of a surrogate objective function. Furthermore, we show that the expected (typical) value of several prediction metrics, and some training metrics, can be computed precisely by computing the location of this saddle point.
In principle, the Cox semi-parametric model can be recovered by an appropriate \emph{data dependent} choice of the time intervals, i.e. by assuming the base hazard rate to be constant between uncensored observations \cite{Breslow_72,Breslow_75}. However, the proof presented in this manuscript is not directly applicable then, since the present hinges on the fact the number of parameters describing the base hazard rate is finite as $n\rightarrow\infty$ (which is not the case for the Cox model). 

The manuscript is structured as follows. In section \ref{sec:2} we introduce the model under study and the related notation; the main theoretical results are presented in section \ref{sec:results}, whilst a brief sketch of the proof can be found in section \ref{sec:proof_sketch} referring to the appendices for the technical details.  We illustrate the agreement of theory and simulations in section \ref{sec:num_exp} via numerical experiments, where we quantify the prediction error by means of the concordance-index and an oracle version of the integrated brier score. Concluding remarks  are in section \ref{sec:conclusion}.  
 
\section{SETTING}
\label{sec:2}

We assume that the time at which a subject fails (i.e. experiences the event) is generated as
\begin{equation}
\label{law}
     Y|\mathbf{X}\sim f_0(.|\mathbf{X}'\bbeta_0) , \qquad \mathbf{X} \sim \mathcal{N}(\bm{0},\bm{I}_{p})\ , 
\end{equation}
for some unknown $f_0 : \mathbb{R}^+ \rightarrow \mathbb{R}^+$ and unknown $\bbeta_0 \in \mathbb{R}^p$, where we indicate with $'$ the scalar product.
The assumption that the covariates are uncorrelated is because we want to focus mostly on the part of the proof that depends on the model (which is a novel application). When the covariates are correlated, the proof is complicated by additional terms that depend on the spectrum of the population covariance matrix and the regularization function, see  for instance \cite{Loureiro_2022}. 
We focus on right censoring, being the most common in applications. In this scenario, instead of observing the actual response $Y$ directly, we observe the event time $T$ and the censoring indicator $\Delta$ defined as follows
\begin{equation}
\label{gen_proc}
    T = \min\{Y,C\} , \qquad \Delta =\bm{1}\big[T<C\big]
\end{equation}
where $C$ is a random variable assuming non-negative values, called the censoring time, i.e. the time at which the subject drops out of the study. We shall further assume that censoring is uninformative, i.e. $Y\perp C| \mathbf{X}$,  and that the censoring is at random and not dependent on the subject's covariates $C\perp\mathbf{X}$.
Given the data generated according to (\ref{law}) and (\ref{gen_proc}), the statistician assumes a proportional hazards model 
\begin{equation}
    \Delta,T|\mathbf{X} \sim \Big(\lambda(T|\bomega)\rme^{\mathbf{X}'\bbeta}\Big)^{\Delta}\rme^{-\Lambda(T|\bomega)\exp{(\mathbf{X}'\bbeta)}} \times f_C(T)^{1-\Delta} S_C(T)^{\Delta}, \quad  \lambda(.|\bomega) := \frac{\rmd }{\rmd t } \Lambda(t|\bomega), 
\end{equation}
where  $C\sim f_C$ and $S_C(x) = \int_{x}^{+\infty} f_c(y) \rmd y$ are, respectively, the density and cumulative distribution function of the censoring risk, $\bomega \in\mathbb{R}^\ell$ are parameters describing the shape of the base hazard rate for the primary risk $\lambda(.|\bomega):\mathbb{R}^{+}\rightarrow \mathbb{R}^+$ and $\Lambda(t|\bomega) := \int_0^{t} \lambda(s|\bomega) \rmd s$ is the corresponding cumulative hazard. 
The piece-wise exponential proportional hazards model postulates that $\lambda$ is a linear combination of piece-wise functions, i.e. zeroth order spline functions, with \emph{user specified knots} $\tau_1, \dots, \tau_{\ell+1}$ \cite{Friedman_82,Kalbfleisch_73} as follows 
\begin{equation}
    \label{def:h}
    \lambda(t|\bomega) := \sum_{k=1}^\ell \psi_k(t)\exp(\omega_k) = \bpsi(t)' \exp(\bomega),  \quad \bomega\in \mathbb{R}^\ell
\end{equation}
with 
\begin{equation}
    \bpsi(t) = \big(\psi_0(t), \dots, \psi_{\ell-1}(t)\big)^\top, \quad \psi_k(t) := \bm{I}[\tau_k < t<\tau_{k+1}], \quad k = 1, \dots, \ell \ .
\end{equation}
The parametric form of the cumulative hazard can be obtained via integration and reads
\begin{equation}
    \label{def:H}
    \Lambda(t|\bomega) := \bm{\Psi}(t)' \exp(\bomega)  
\end{equation}
where
\begin{equation}
    \bm{\Psi}(t) = \big(\Psi_0(t), \dots, \Psi_{\ell - 1}(t)\big)^\top, \quad  \Psi_k(t) := \bm{I}[t > \tau_k]\min\{t-\tau_k,\tau_{k+1}-\tau_k\} \ .
\end{equation}
We notice that this is a special case of a B-spline parameterization of the baseline hazard function.
The model is, in principle, arbitrary flexible in the sense that increasing the number of intervals allows fitting more and more precisely the shape of the \say{true} (unknown) baseline hazard rate. 
It has been noticed in literature \cite{Harrell_2001} that the locations of the knots do not generally impact on the quality of the fit, but the number of knots is crucial. In this manuscript we take a penalized splines (P-spline) approach, use several equispaced knots and penalize the respective coefficients via a ridge-like regularizer. To keep the setting simple we consider a simple uniform ridge regularization (although smoothness penalties based on finite differences might be easily considered in the present setting), i.e. the model is fitted to the data  $\{(\Delta_1, T_1, \mathbf{X}_1), \dots, (\Delta_n, T_n, \mathbf{X}_n)\}$, generated according to (\ref{law}) and (\ref{gen_proc}), by minimizing the following objective function
\begin{equation}
\label{def : objective}
    \mathfrak{L}_n(\bomega,\bbeta) = \frac{1}{n}  \sum_{i=1}^n \Big\{ g(\mathbf{X}_i'\bbeta,\bomega,T_i,\Delta_i)- \Delta_i \log \lambda(t_i|\bomega) \Big\} + \frac{1}{2}\eta\|\bbeta\|^2 + \frac{1}{2} \alpha \|\bomega\|^2 
\end{equation}
with the function $g$, defined as 
\begin{equation}
    g(x,\bomega,T,\Delta) = \Lambda(T|\bomega)\rme^{x} - \Delta x \ .
\end{equation}
The piece-wise exponential model lets us control the \say{complexity} of the parametrization of the base hazard rate $h$ via the number of intervals $\ell+1$  (user specified) in which the latter is assumed to be constant. This is very convenient when the sample size $n$ is large (as we shall assume): in this case the number of parameters defining the survival function is only $\ell$ numbers, and not $n$ as for Cox model \cite{Kvamme_2019,Kvamme_2021}.
The ridge penalization is often introduced in practice to improve numerical stability by enforcing strong convexity of $\mathfrak{L}_n(\bomega,\bbeta)$ in its arguments, so the minimizer always exists unique.

\section{TECHNICAL RESULTS}
\label{sec:results}
We need to introduce some definitions from convex analysis. In particular, we will need the following.
\begin{definition}[Moreau envelope and proximal operator]
    Given a convex function $f:\mathbb{R}\rightarrow\mathbb{R}$ the Moreau envelope function evaluated at $x$ with parameter $\nu \in \mathbb{R}^+$, is defined as 
    \begin{equation}
        \mathcal{M}_{f (.)}(x,\nu) := \underset{y}{\min} \Big\{\frac{1}{2\nu}(y-x)^2 + f(y)\Big\} \ .
    \end{equation}
    The point at which the minimum is attained is called the proximal operator of $f$
    \begin{equation}
        \prox_{f(.)}(x,\nu) := \underset{y}{\arg\min} \Big\{\frac{1}{2\nu}(y-x)^2 + f(y)\Big\} \ .
    \end{equation}
\end{definition}
These functions are frequently encountered in convex analysis problems as they possess several desirable properties, e.g.the set of minimizers of $\mathcal{M}_{f (.)}(.,\nu>0), \ \forall \nu>0$ coincides with that of the minimizers of $f(.)$ \cite{Rockafellar_1997}. 

The main technical results of the manuscript are established in the following theorems and corollaries. In short the asymptotic value of several quantities of interest derived from the Penalized Maximum Likelihood estimator can be obtained by studying an asymptotically equivalent problem, which is low dimensional and hence quickly and easily solvable (in particular for large values of $p$), instead of studying directly the original high dimensional problem (\ref{def : objective}).
\begin{theorem}[ASYMPTOTICALLY EQUIVALENT SCALAR OPTIMIZATION PROBLEM]
\label{theorem:equivalent_prob}
    Let $\zeta := \lim_{n\rightarrow \infty} p(n)/n$,  $\zeta \in \mathbb{R}^+$ and assume that the data are generated as in (\ref{law}, \ref{gen_proc}).
    Then, for any $\epsilon >0$,
    \begin{equation}
    \label{eq : asymptotic determinist equivalence}
         P\Bigg[\ \bigg|\underset{\bomega,\bbeta}{\min}\  \mathfrak{L}_n(\bomega,\bbeta) - \underset{\bomega,w,v}{\min} \ \underset{\phi\geq0}{\max} \ \underset{\tau>0}{\inf} \ \mathcal{L}(\bomega,w,v,\phi,\tau) 
         \bigg|>\epsilon \ \Bigg] \xrightarrow[n\rightarrow \infty]{} 0
    \end{equation}
    where 
    \begin{eqnarray}
    \label{ASOP}
        \mathcal{L}(\bomega,w,v,\phi,\tau) &:=& \mathbb{E}_{T,Z_0,Q}\Big[\mathcal{M}_{g(.,\bomega,\Delta,T)}\big(wZ_{0}  +v Q,\tau/\phi\big)- \Delta \log \lambda(T|\bomega)\Big] +   \nonumber \\
        &+&\phi\big(\tau/2- v\sqrt{\zeta}\big) +\frac{1}{2}\eta (v^2 + w^2)  + \frac{1}{2} \alpha \|\bomega\|^2
    \end{eqnarray}
    with  $Z_0,Q \sim \mathcal{N}(0,1)$, $Z_0\perp Q$.
\end{theorem}
The convergence in probability of the optimum value of the objective function implies the following.
\begin{theorem}[REPLICA SYMMETRIC EQUATIONS]
\label{theorem:rs_eqs}
Let $\hat{\bbeta}_n,\hat{\bomega}_n$ be the (unique) minimizer of $\mathfrak{L}_n$ as defined in (\ref{def : objective}), then
    \begin{eqnarray}
        \label{w_n}
        &&\hat{w}_n := \frac{\bbeta_0'\hat{\bbeta}_n}{\|\bbeta_0\|}\xrightarrow[n\rightarrow \infty]{P} w_{\star}\\
        \label{v_n}
        &&\hat{v}_n :=\big\|\mathbf{P}_{\perp\bbeta_0}\hat{\bbeta}_{n}\big\|= \Big\|\Big(\bm{I} - \frac{\bbeta_0\bbeta_0'}{\|\bbeta_0\|^2}\Big)\hat{\bbeta}_n\Big\| \xrightarrow[n\rightarrow \infty]{P} v_{\star}\\
        \label{lambda_n}  
        &&\hat{\bomega}_n \xrightarrow[n\rightarrow \infty]{P} \bomega_{\star} \ .
    \end{eqnarray}
    The values $\bomega_{\star},w_{\star},v_{\star},\tau_{\star},\phi_{\star}$ identify the saddle point of $\mathcal{L}$ and solve the following set of self-consistent equations
    \begin{eqnarray}
    \label{rs1}
        v^2 \zeta &=& \mathbb{E}_{T,Z_0,Q}\Big[\big\|\hat{\xi} - wZ_0  - vQ\big\|^2\Big] \\
    \label{rs2}
        w(1+\eta\tau/\phi)&=& \mathbb{E}_{T,Z_0,Q}\Big[Z_0\hat{\xi}\Big]\\
    \label{rs3}
         v(1-\zeta +\eta \tau/\phi) &=&  \mathbb{E}_{T,Z_0,Q}\Big[Q\hat{\xi}\Big]\\
    \label{rs4}
         \tau &=& v\sqrt{\zeta}\\
    \label{rs5}
        \omega_k &=& \frac{1}{\eta \alpha} \mathbb{E}\Big[\Delta\psi_k(T)\Big] - W_0\bigg(\frac{1}{\eta \alpha} \mathbb{E}\Big[\rme^{\hat{\xi}}\psi_k(T)\Big] \exp\Big\{ \frac{1}{\eta \alpha} \mathbb{E}\Big[\Delta \Psi_k(T)\Big]\Big\}\bigg), \ k = 1, \dots, \ell , 
    \end{eqnarray}
    where  $W_0$ is the (real branch of) Lambert W - function \cite{Corless_1996}, defined as the (real) solution of the equation $W_0(x)\exp\{W_0(x)\} = x$ and 
    \begin{eqnarray}
    \label{prox}
        \hat{\xi}(Z_0,Q,T) &:=&\prox_{g(.,\bomega,\Delta,T)}(wZ_0 + v Q,\tau/\phi)=\nonumber \\
        &=&wZ_0 + v Q +\Delta \tau/\phi - W_0\Big(\tau \Lambda(T|\bomega)\rme^{\Delta\tau/\phi +wZ_0 + v Q}\Big)   \ .
    \end{eqnarray}
\end{theorem}
The theorems (\ref{theorem:equivalent_prob},\ref{theorem:rs_eqs}) above are a generalization of previous results \cite{Thrampoulidis_2018,Loureiro_2022} to the piece-wise exponential model. We briefly sketch the proof ideas in the next section, and we delegate the details to the supplementary material \ref{appendix:pointwise},\ref{appendix:saddle_convergence},\ref{appendix:conv_minimizer}.
A consequence of the theorem above is the following.
\begin{corollary}[SURROGATE FOR OUR OF SAMPLE LINEAR PREDICTOR]
\label{corollary:surr_lp_oos}
     Let $f:\mathbb{R}^{1+l} \rightarrow\mathbb{R}$ and $\tilde{\mathbf{X}}$ a newly generated covariate vector
    \begin{equation}
        f(\tilde{\mathbf{X}}'\hat{\bbeta}_n,\hat{\bomega}_n) \xrightarrow[n\rightarrow \infty]{d} f\big(w_{\star}Z_0+v_{\star}Q,\bomega_{\star}\big)
    \end{equation}
\end{corollary}
\begin{proof}
    By Slutsky's Lemma \cite{Van_der_Vaart_2000} we have that for a \say{fresh} covariate vector $\tilde{\mathbf{X}}\sim \mathcal{N}(\bm{0},\bm{I}_{p})$
    \begin{equation}
        \tilde{\mathbf{X}}'\hat{\bbeta}_n = \frac{\tilde{\mathbf{X}}'\bbeta_0}{\|\bbeta_0\|} \frac{\bbeta_0'\hat{\bbeta}_n}{\|\bbeta_0\|} + \tilde{\mathbf{X}}'(\bm{I}_p - \frac{\bbeta_0\bbeta_0'}{\|\bbeta_0\|^2})\hat{\bbeta}_n \overset{d}{=}\tilde{Z}_0 w_n + \tilde{Q} v_n \xrightarrow[n\rightarrow \infty]{d}  \tilde{Z}_0 w_{\star} + \tilde{Q} v_{\star},\quad  \tilde{Z}_0,\tilde{Q} \sim \mathcal{N}(0,1) 
    \end{equation}
    because of the convergence in probability of $w_n,v_n$ (\ref{w_n},\ref{v_n}).
    The conclusion follows by the continuous mapping theorem \cite{Van_der_Vaart_2000} (page 7 theorem 2.3).
\end{proof}
The corollary above allows to precisely compute prediction metrics like the ones that we study in section (\ref{sec:num_exp}).

\section{SKETCH OF THE PROOF OF MAIN THEOREM 1}
\label{sec:proof_sketch}
Introducing Lagrange multipliers $\bphi$, one can equivalently re-write the minimization problem as a saddle point problem as follows 
\begin{eqnarray}
    \ell_n(\mathbf{X},\mathbf{T},\eta) &=& \underset{\bomega,\bbeta,\bxi}{\min}\ \underset{\bphi}{\sup} \Big\{\frac{1}{n} \sum_{i=1}^n \Big[ g(\xi_i,\bomega,\Delta_i,T_i) - \Delta_i \log \lambda(T_i|\bomega)\Big] - \bphi'\big(\bxi - \mathbf{X}\bbeta\big)  +\nonumber \\
    &+& \frac{1}{2}\eta \|\bbeta\|^2  + \frac{1}{2}\alpha \|\bomega\|^2\Big\} \ .
\end{eqnarray}
Notice that $\mathbf{X}\bbeta = \mathbf{X}\mathbf{P}_{\bbeta_0}\bbeta + \mathbf{X}\mathbf{P}_{\perp\bbeta_0}\bbeta  \overset{d}{=}  \mathbf{Z}_0 \beta_{\parallel} + \Tilde{\mathbf{X}}\bbeta_{\perp}$, where $\overset{d}{=}$ indicates equality in distribution with $\beta_{\parallel}:=\frac{\bbeta_0'\bbeta}{\|\bbeta_0\|}$, $\bbeta_{\perp}:= \mathbf{P}_{\perp\bbeta_0}\bbeta = (\bm{I} - \bbeta_0\bbeta_0'/\|\bbeta_0\|^2)\bbeta$, $ \mathbf{Z}_0:= \mathbf{X}\bbeta_0/\|\bbeta_0\|\sim \mathcal{N}(\bm{0},\bm{I}_n)$ and $\Tilde{\mathbf{X}} := \mathbf{X}\mathbf{P}_{\bbeta_0}\ (\Tilde{\mathbf{X}})_{i,j}\sim  \mathcal{N}(0,1)$ and $\Tilde{\mathbf{X}}\perp \mathbf{T}$.
Using this fact, we have reduced the original problem into the form 
\begin{eqnarray}
\label{PO}
    \ell_n(\mathbf{X},\mathbf{T},\eta ) &\overset{d}{=}& \underset{\bomega,\bbeta,\bxi}{\min} \ \underset{\bphi}{\sup} \bigg\{\frac{1}{n} \sum_{i=1}^n \Big[g(\xi_i,\bomega,\Delta_i,T_i) - \Delta_i \log \lambda(T_i|\bomega)\Big] - \bphi'\big(\bxi -\mathbf{Z}_0 \beta_{\parallel} - \Tilde{\mathbf{X}}\bbeta_{\perp}\big)  + \nonumber \\
    &+&\frac{1}{2}\eta \|\bbeta\|^2 + \frac{1}{2}\alpha \|\bomega\|^2 \bigg\}
\end{eqnarray}
that can be attacked with the Convex Gaussian Min-Max theorem (CGMT), first introduced in \cite{Thrampoulidis_2015} as a generalization of Gordon's Gaussian comparison inequalities \cite{Gordon_85}. The CGMT is reported without proof below for the reader's convenience. We refer to \cite{Thrampoulidis_2015,Thrampoulidis_2018} for a detailed proof.
\begin{theorem}[CONVEX GAUSSIAN MIN-MAX THEOREM]
\label{CGMT}
    Let $\mathcal{S}_{\by}\subset \mathbb{R}^n$, $\mathcal{S}_{\bz}\subset \mathbb{R}^p$ be compact and convex sets, $\psi$ be continuous and convex-concave on $\mathcal{S}_{\bz}\times \mathcal{S}_{\by}$, and $\mathbf{X}\in\mathbb{R}^{n\times p},\mathbf{Q}\in\mathbb{R}^{n},\mathbf{G}\in \mathbb{R}^p$ all have entries i.i.d. standard normal
    \begin{eqnarray}
        \Phi(\mathbf{X}) &:=& \underset{\by \in \mathcal{S}_{\by}}{\min} \ \underset{\bz \in \mathcal{S}_{\bz}}{\max} \Big\{\by'\mathbf{X}\bz + \psi(\bz,\by)\Big\}\nonumber\\
        \phi(\mathbf{G},\mathbf{Q}) &:=& \underset{\by \in \mathcal{S}_{\by}}{\min} \ \underset{\bz \in \mathcal{S}_{\bz}}{\max} \Big\{\|\by\|\mathbf{G}'\bz + \|\bz\| \mathbf{Q}'\by + \psi(\bz,\by)\Big\} \nonumber
    \end{eqnarray}
    Then
    \begin{equation}
        \forall \mu\in \mathbb{R},\  t \in \mathbb{R}^+, \  P\Big[ \Big|\Phi(\mathbf{X})-\mu\Big| > t \Big] \leq  2 P\Big[ \Big|\phi(\mathbf{G}, \mathbf{Q}) -\mu\Big| \geq  t \Big] \ .
    \end{equation}
    Furthermore, let $\mathcal{S}\subset \mathcal{S}_{\by}$ an open subset and $\mathcal{S}^c := \mathcal{S}_{\by}\setminus \mathcal{S}$. Denote $\phi_{\mathcal{S}^c}(\mathbf{G},\mathbf{Q})$ the optimal cost
    of the surrogate process, when the minimization is now constrained over $\mathcal{S}^c$. If there exist constants $\bar{\phi}<\bar{\phi}_{\mathcal{S}^c}$, such that $\phi(\mathbf{G},\mathbf{Q})\xrightarrow[n\rightarrow \infty]{P}\bar{\phi}$, $\phi_{\mathcal{S}^c}(\mathbf{G},\mathbf{Q})\xrightarrow[n\rightarrow \infty]{P}\bar{\phi}_{\mathcal{S}^c}$, then, denoting with $\hat{y}$ the value of $y$ at the saddle point, we have 
    \begin{equation}
    \label{asymptotic_cgmt}
        \lim_{n\rightarrow \infty} P\Big[\hat{\by} \in \mathcal{S}\Big] = 1\ .
    \end{equation}
\end{theorem}
Theorem \ref{CGMT} above tells us that if one is able to prove that $\phi$ converges in probability to some deterministic value, then the same is true for $\Phi$.
The CGMT applies to min-max problems over compact, convex sets. Hence, we \say{artificially} restrict the saddle point problem (\ref{PO}) onto compact, convex sets. 
Intuition suggests that if a saddle point exists and the set is sufficiently large, then there is not going to be any difference between the bounded and unbounded problem. 
Hence, from now on the min over $\bbeta,\bxi$ and the max over $\bphi$ operations are understood over convex, compact sets (so that they actually exist).
Following Theorem \ref{CGMT}, we consider an auxiliary optimization problem
\begin{eqnarray}
\label{AO0}
    \Tilde{\ell}_n(\mathbf{Q},\mathbf{G},\mathbf{T})&=& \underset{\bomega,\bbeta,\bxi}{\min}\ \underset{\bphi}{\max}\bigg\{ \frac{1}{n} \sum_{i=1}^n \Big[ g(\xi_i,\bomega,\Delta_i,T_i)  - \Delta_i \log \lambda(T_i|\bomega) \Big]  + \nonumber \\
    &-&  \bphi'\big(\bxi - \beta_{\parallel}\mathbf{Z}_0  - \|\bbeta_{\perp}\|\mathbf{Q}\big) + \bbeta_{\perp}'\mathbf{G}\|\bphi\| + \frac{1}{2}\eta \|\bbeta\|^2 + \frac{1}{2}\alpha\|\bomega\|^2 \bigg\} 
\end{eqnarray}
with $\mathbf{G}\sim \mathcal{N}(\bm{0},\bm{I}_{(p-1)})$ and $\mathbf{Q}\sim \mathcal{N}(\bm{0},\bm{I}_{n})$, $\mathbf{G}\perp\mathbf{Q}$ and $\mathbf{G},\mathbf{Q}\perp \mathbf{T},\mathbf{Z}_0$.
We can optimize over the direction of $\bphi$ at fixed length $\|\bphi\| = \phi$
\begin{eqnarray}
\label{AO1}
    \Tilde{\ell}_n(\mathbf{Q},\mathbf{G},\mathbf{T}) &=& \underset{\bomega,\bbeta,\bxi}{\min}\ \underset{\phi\geq 0}{\max}\bigg\{ \frac{1}{n} \sum_{i=1}^n \Big[ g(\xi_i,\bomega,\Delta_i,T_i)  - \Delta_i \log \lambda(T_i|\bomega)\Big]+  \nonumber\\
    &+&  \phi\big\|\bxi - \beta_{\parallel}\mathbf{Z}_0  - \|\bbeta_{\perp}\|\mathbf{Q}\big\| + \bbeta_{\perp}'\mathbf{G}\phi  +\frac{1}{2}\eta \|\bbeta\|^2 + \frac{1}{2}\alpha \|\bomega\|^2 \bigg\} \ .
\end{eqnarray}
The problem above depends on $\bbeta_{\perp}$ via $\|\bbeta_{\perp}\|$ and $\cos \theta $, with $\theta$ the angle between $\bbeta_{\perp}$ and $\mathbf{G}$. Hence, it is not guaranteed to be convex-concave (as $cos$ is not convex on the whole interval $[0,2\pi]$), but it has been shown in \cite{Thrampoulidis_2018} (page 22 point 6) that (\ref{AO1}) can be used in place of (\ref{AO0}) in the CGMT as $n\rightarrow\infty$, i.e. the min-max order can be \say{swapped} asymptotically. At this point, the minimization over the direction of $\bbeta_{\perp}$ at fixed $\|\bbeta_{\perp}\| = v$ yields
\begin{eqnarray}
     \Tilde{\ell}_n(\mathbf{Q},\mathbf{G},\mathbf{T}) &=& \underset{\bomega,w,v}{\min} \ \underset{\phi}{\max}\ \underset{\bxi}{\min}\bigg\{ \frac{1}{n} \sum_{i=1}^n  \Big[ g(\xi_i,\bomega,\Delta_i,T_i) - \Delta_i \log \lambda(T_i|\bomega)\Big]  +\nonumber\\
     &+& \phi\big\|\bxi - w\mathbf{Z}_0  -v\mathbf{Q}\big\| - v \phi  \|\mathbf{G}\| + \frac{1}{2}\eta (v^2 + w^2) + \frac{1}{2}\alpha \|\bomega\|^2\bigg\} \ ,
\end{eqnarray}
where we also defined $w := \bbeta_{\parallel}$.
Following \cite{Thrampoulidis_2018}, we use the variational representation $\|\mathbf{y}\| = \underset{\tau\geq 0}{\inf} \Big\{\frac{1}{2}\Big(\tau\|\mathbf{y}\|^2 + \frac{1}{\tau}\Big) \Big\}$ for the norm.
Then 
\begin{eqnarray}
    \Tilde{\ell}_n(\mathbf{Q},\mathbf{G},\mathbf{T}) &=& \underset{\bomega,w,v}{\min} \ \underset{\phi}{\max} \ \underset{\bxi,\tau}{\inf} \bigg\{ \frac{1}{n} \sum_{i=1}^n \Big[ g(\xi_i,\bomega,\Delta_i,T_i) - \Delta_i \log \lambda(T_i|\bomega) \Big]  + \frac{1}{2\tau}\big\|\bxi - w\mathbf{Z}_0  - v\mathbf{Q}\big\|^2 + \nonumber \\
    &+&  \phi(\tau/2- v\|\mathbf{G}\|)  + \frac{1}{2}\eta (v^2 + w^2)+ \frac{1}{2}\alpha \|\bomega\|^2\bigg\} \ .
\end{eqnarray}
Taking the re-scaling $\tau\rightarrow \sqrt{n}\tau, \ \phi\rightarrow\phi/\sqrt{n}$ we recognize the (averaged) Moreau envelope (as defined in the main previous section)
\begin{eqnarray}
    &&\frac{1}{n}  \underset{\bxi}{\min}\Big\{  \sum_{i=1}^n  g(\xi_i,\bomega,\Delta_i,T_i)   + \frac{\phi}{2\tau}\big\|\bxi - w\mathbf{Z}_0  - v\mathbf{Q}\big\|^2\Big\} = \frac{1}{n} \sum_{i=1}^n \underset{\xi}{\min}\Big\{    g(\xi,\bomega,\Delta_i,T_i)   + \frac{\phi}{2\tau}\big( \xi - w Z_{0,i}  - v Q_i \big)^2\Big\}  \nonumber\\
    &&=\hspace{1cm}\frac{1}{n} \sum_{i=1}^n \mathcal{M}_{g(.,\bomega,\Delta_i,T_i)  }\big(wZ_{0,i} +  vQ_i,\tau/\phi\big)   \ .
\end{eqnarray}
Hence, we are finally left with the saddle point problem
\begin{equation}
    \Tilde{\ell}_n(\mathbf{Q},\mathbf{G},\mathbf{T}) =\underset{\bomega,w,v}{\min} \ \underset{\phi}{\max} \ \underset{\tau}{\inf} \ \mathcal{L}_n(\bomega,w,v,\phi,\tau)
\end{equation}
where
\begin{eqnarray}
    \mathcal{L}_n(\bomega,w,v,\phi,\tau) &=& \frac{1}{n} \sum_{i=1}^n \Big[\mathcal{M}_{g(.,\bomega,\Delta_i,T_i)  }\big(wZ_{0,i} +  vQ_i,\tau/\phi\big) - \Delta_i \log \lambda(T_i|\bomega)\Big]  + \phi\big(\tau/2- v\|\mathbf{G}\|/\sqrt{n}\big)+\nonumber\\
    &+& \frac{1}{2}\eta (v^2 + w^2)+ \frac{1}{2}\alpha \|\bomega\|^2 \ .\nonumber
\end{eqnarray}
The weak law of large numbers and a concentration argument for $\|\mathbf{G}\|$ imply, see appendix \ref{appendix:pointwise}, that 
\begin{equation}
\label{point_conv}
    \mathcal{L}_n(\bomega,w,v,\phi,\tau)  \xrightarrow[n \rightarrow \infty]{P} \mathcal{L}(\bomega,w,v,\phi,\tau) 
\end{equation}
pointwise for $\|\bomega\|\leq C_{\bomega},0\leq w\leq C_{\bbeta},0\leq v\leq C_{\bbeta}, \phi\geq 0$ and $\tau>0$, where 
\begin{eqnarray}
    \mathcal{L}(\bomega,w,v,\phi,\tau) &:=& \mathbb{E}_{T,Z_0,Q}\Big[\mathcal{M}_{g(.,\bomega,\Delta,T)  }\big(wZ_{0}  +v Q,\tau/\phi\big)-\Delta \log \lambda(T|\bomega)\Big] + \nonumber \\
    &+&\phi\big(\tau/2- v\sqrt{\zeta}\big) + \frac{1}{2}\eta (v^2 + w^2) + \frac{1}{2}\alpha \|\bomega\|^2 \ ,
\end{eqnarray}
with $Q\sim \mathcal{N}(0,1) $, $ Z_0 \sim \mathcal{N}(0,1)$, $Z_0\perp Q$.
The pointwise convergence in probability above, together with the fact that the functions $\mathcal{L}_n,\mathcal{L}$ are concave in $\phi$ and convex in $\bomega,w,v,\xi,\tau$ imply, see appendix \ref{appendix:saddle_convergence}, that
\begin{equation}
     \underset{\bomega,w,v}{\min} \ \underset{\phi>0 }{\max} \ \underset{\tau> 0}{\inf} \  \mathcal{L}_n(\bomega, w,v,\phi,\tau)  \xrightarrow[n\rightarrow \infty]{P} \underset{w,v}{\min} \ \underset{\phi>0}{\max} \ \underset{\tau>0}{\inf} \ \mathcal{L}(\bomega, w,v,\phi,\tau)  \ .
\end{equation}
By the CGMT, we then have the desirata, i.e. Theorem \ref{theorem:equivalent_prob}.
Furthermore the pointwise convergence (\ref{point_conv}) and the strict convexity of $\mathcal{L}$ in $w,v,\bomega$ implies the convergence in probability of the minimizer (\ref{w_n},\ref{v_n},\ref{lambda_n}) via (\ref{asymptotic_cgmt}) see appendix \ref{appendix:conv_minimizer}.

\section{NUMERICAL EXPERIMENTS}
\label{sec:num_exp}
In the following we simulate the model under study and compare various quantities, such as goodness of fit and prediction metrics, against the theory for different values of the regularizer $\eta$ which controls the amount of ridge shrinking on $\hat{\bbeta}_n$.  
The data simulations are carried out as follows. We take the sample size $n=400$ fixed and vary the number of covariates $p$ via $\zeta$ as $p = \zeta n$. The latent survival times are generated from a Log-logistic proportional hazard model
\begin{equation}
    Y_i|\mathbf{X}_i \sim   -\frac{\rmd }{\rmd  t }  S_0(t|\mathbf{X}_i), \quad  S_0(t|\mathbf{X}_i) = \exp\{- \Lambda_0(.)\exp(\mathbf{X}_i'\bbeta_0)\}, \ \Lambda_0(t) := \log \Big(1 +  t^2/2\Big)
\end{equation}
where $\mathbf{X}_i\sim \mathcal{N}(\bm{0},\bm{I}_p)$. Censoring is taken to be uniform between $\tau_1 =1$ and $\tau_2 = 3$. With these choices, the expected fraction of censored events is $40\%$, i.e. only $60\%$ of the $n$ subjects experiences the event on average. We sample $\bbeta_0 \sim \mathbb{S}_{p-1}$, i.e. the true vector of association parameter is drawn at random from the unit sphere in $\mathbb{R}^p$. Notice that our choice of $\bbeta_0$ does not imply any loss in generality since the distribution of $\mathbf{X}$ and the ridge penalty are rotationally invariant, hence any $\bbeta_0$ with the same length will yield the same statistics for the population of $Y_i$.

For each value of $\zeta$ (or equivalently $p$ at fixed $n$) we simulated $50$ datasets and computed the penalized maximum likelihood estimator $\hat{\bbeta}_n,\hat{\bomega}_n$ by numerical minimization of (\ref{def : objective}) along a regularization path for $\eta$ at fixed $\alpha= 0.01$. In all the plots that follow, the markers are empirical averages, while the error bars are empirical standard deviations computed over these realizations. 
We chose to use $11$ time points $\tau_{k=1}^\ell$ equispaced between $0$ and $3$ (the end of the study) as the knots of the piece-wise parametrization of the hazard rate, this allows for a flexible approximation of the hazard rate.
For the solution of the RS equations, we computed numerically the solution via fixed point iteration, to a tolerance of $1.0^{-8}$. The expectations in (\ref{rs1},\ref{rs2},\ref{rs3},\ref{rs5}) are approximated as population averages, with a population size $m= 2\cdot 10^{3}$.

Before proceeding to examine the results of the numerical experiments, we first need to introduce which metrics were employed to score the predictive ability of the model. The following subsection deals with this.

\subsection{Evaluation metrics for Survival Predictions}
In order to evaluate prediction accuracy in survival analysis, it is generally advised to always consider at least two metrics: one for the calibration and another for the discrimination ability of the model. Calibration, quoting from \cite{Steyenberg_2010}, \say{refers to the agreement between observed outcome and predictions}. Discrimination is the ability of the model to separate individuals with different risks scores. The rationale being that a good model should assign shorter survival times to subjects with higher risk score.  
To evaluate the discriminative ability of the model, we used Harrell's c index \cite{Harrell_82,Harrell_2001}
\begin{equation}
    HC_n =\frac{ \sum_{i=1}^n \Delta_i \sum_{j=1}^n \Theta(T_j-T_i) \bm{1}\Big[f(\mathbf{X}_j)>f(\mathbf{X}_i)\Big]}{ \sum_{i=1}^n \Delta_i \sum_{j=1}^n \Theta(T_j-T_i) } \ .
\end{equation}
This quantity is close to one if the model has perfect discrimination ability. Random guessing would give a c-index of $0.5$. 
Instead of evaluating the calibration of the model directly via the Brier Score, we compute the following,
\begin{equation}
\label{IBS_ideal}
    IBS_{ideal} = \mathbb{E}_{\mathbf{X}}\bigg[ \int \Big(\hat{S}(t|\mathbf{X}) -S_0(t|\mathbf{X})\Big)^2 \rmd t\bigg]  \ .
\end{equation}
This quantity measures the integrated Mean Squared Error of the estimated survival function with respect to the true one.
The \say{ideal} in underscore is due to the fact that the quantity above is what any estimator of the Integrated Brier Score (IBS) aims ideally to estimate. The ideal IBS (\ref{IBS_ideal}) can only be computed for simulated data because in that case we know the $S_0(.|\mathbf{X})$ used to generate the data. Notice that (\ref{IBS_ideal}) can only be used to score the model relative to another one, hence we prefer to use the ratio
\begin{equation}
\label{def : R_ibs}
    R_{IBS}= \frac{\mathbb{E}_{\mathbf{X}}\bigg[ \int \Big(\hat{S}(t|\mathbf{X}) -S_0(t|\mathbf{X})\Big)^2 \rmd t\bigg]}{\mathbb{E}_{\mathbf{X}}\bigg[ \int \Big(\hat{S}_{null}(t) -S_0(t|\mathbf{X})\Big)^2 \rmd t\bigg]}
\end{equation}
where $\hat{S}_{null}$ is the estimated survival function when no covariates are included in the model.
Thanks to corollary (\ref{corollary:surr_lp_oos}) we can easily compute the metrics above for the test set via the solution of the RS equations (see theorem \ref{theorem:rs_eqs}).

\subsection{Comparison of theory and simulations}
Figures (\ref{fig:1},\ref{fig:2}) show the comparison between the theoretical values obtained by solving the Replica Symmetric equations (i.e. the non-linear system in corollary \ref{theorem:rs_eqs}) as solid lines and the results obtained by the simulations as markers with error bars. As expected, because of the convergence in probability established in (\ref{theorem:rs_eqs}), the solid line describes accurately the position of the markers for different values of the ratio $\zeta$.   
We notice that, as expected, the test c-index is virtually independent of the penalization strength as one can see in (\ref{fig:2a}) and its maximal value decreases with $\zeta$.
Whilst from figure (\ref{fig:2b}) we deduce that as $\zeta$ increases, the minimum of $R_{IBS}$ is attained at a larger value of $\eta$, i.e. more regularization is needed just to do slightly better than the null model. 

\begin{figure}[t]
\begin{subfigure}{.47\textwidth}
  \includegraphics[width=\linewidth]{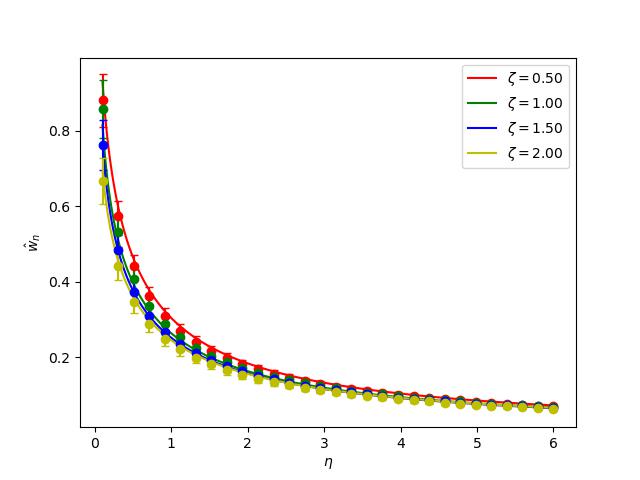}
  \caption{}
  \label{fig:1a}
\end{subfigure}
\hfill
\begin{subfigure}{.47\textwidth}
\includegraphics[width=\linewidth]{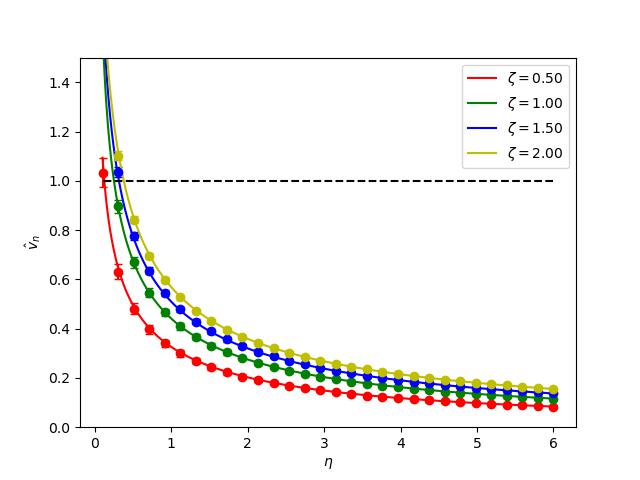}
  \caption{}
  \label{fig:1b}
\end{subfigure}
\caption{Simulated data (markers and error bars) against the theory (solid lines). Figures (\ref{fig:1a},\ref{fig:1b}) show the value of $\hat{w}_n$ and $\hat{v}_n$ defined in (\ref{w_n}, \ref{v_n}) along a regularization path for $\eta \in (0.1, 6)$ with $\alpha = 0.01$.}
\label{fig:1}
\end{figure}

\begin{figure}[t]
\begin{subfigure}{.47\textwidth}
  \includegraphics[width=\linewidth]{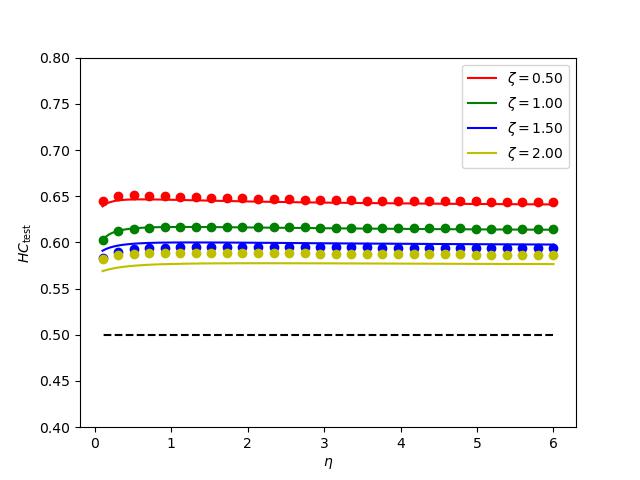}
  \caption{}
  \label{fig:2a}
\end{subfigure}
\hfill
\begin{subfigure}{.47\textwidth}
  \includegraphics[width=\linewidth]{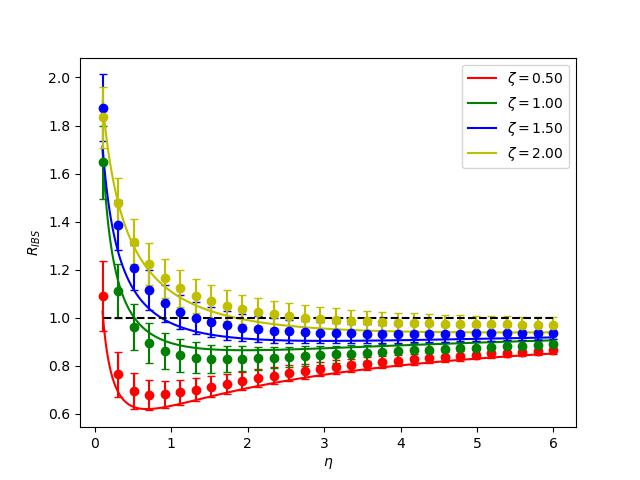}
  \caption{}
  \label{fig:2b}
\end{subfigure}
\caption{Simulated data (markers and error bars) against the theory (solid lines): (left) the Test c-index; (right)  $R_{IBS}$ as defined in (\ref{def : R_ibs}), along a regularization path for $\eta \in (0.1, 6)$ with $\alpha = 0.01$.
The error bars for the figure (\ref{fig:2a}) have been removed to aid visualization.}
\label{fig:2}
\end{figure}

\section{CONCLUSION}
\label{sec:conclusion}
In conclusion, we have applied the Gaussian Convex Min Max theorem to a flexible parametric model for survival data, i.e. the piece-wise exponential model. 
With the theoretical guarantees obtained, we investigated the effect of the ridge penalization onto the predictive ability of the model. 
This represents a first step into the full formalization of a set of recent heuristic results obtained for the Cox model via the replica method of statistical physics \cite{Coolen_17,Massa_2024}, in the sense that our future goal will be to generalize the present proof to the Cox semi parametric model.
While the present work focuses on the piece-wise exponential model, it is more generally a prototype for any parametric Survival analysis model (although the precise applications of various theorems to specific models need to be checked). As such, it puts on a firm ground previous heuristics on the subject \cite{Massa_22}. Further work is required to rigorously establish more recent heuristic results for the semi-parametric Cox model \cite{Massa_2024}. 

The proof technique used here relies completely on the assumption of Gaussian covariates, since it hinges on the Convex Gaussian Min Max theorem \cite{Thrampoulidis_2015}. What is intriguing is that the replica methods requires only that the linear predictor follows (asymptotically) a Normal law. So we expect these results to hold in more general scenarios (e.g. Sub-Gaussian covariates) than the admittedly limited one studied here (Gaussian covariates). It would be ideal to be able to put these several conjectures on a firm theoretical basis, and indeed recent investigations seek to do this, at least in some specific settings \cite{Han_23,Hu_2020,Montanari_2022}.

\section*{Acknowledgements}
The author is grateful to M. A. Jonker for useful discussions,  and to the anonymous referee that contributed to improve the exposition of this manuscript.  
\section*{Supporting Information}
Additional simulations, together with the routines used to compute the estimators, solve the Replica Symmetric equations and plot the figures of the paper are available in GitHub at \url{https://github.com/EmanueleMassa/SMSA_2024}.

\bibliographystyle{unsrt}  
\bibliography{references} 

\appendix

\section{PROPERTIES OF THE MOREAU ENVELOPE FOR THE PIECEWISE EXPONENTIAL MODEL}
In the following  we assume that $\|\bomega\|\leq C_{\\bomega}$, $w,v\leq C_{\bbeta}$ and $\tau>0$, furthermore $\tau_1<\tau_2<\dots<\tau_{\ell +1}<\infty$.

\begin{proposition}[INTEGRABILITY]
\label{integrability}
     The random function  $(\bomega,w,v,\tau) \rightarrow \mathcal{M}_{g(.,\bomega,\Delta,T)}(wZ_0 + vQ,\tau) -\Delta \log \lambda(T|\bomega)$ is absolutely integrable.
\end{proposition}
\begin{proof}
By definition,
\begin{equation}
   \mathcal{M}_{g(.,\bomega,\Delta,T)}(wZ_0 + vQ,\tau)\leq \frac{1}{2\tau} \big(wZ_0 + vQ \big)^2 + g(0,\bomega,\Delta,T) \ .
\end{equation}
hence 
\begin{equation}
   \underset{x}{\min} \ g(x,\bomega,\Delta,T) \leq \mathcal{M}_{g(.,\bomega,\Delta,T)}(wZ_0 + vQ,\tau)\leq  \frac{\lambda}{2\tau} \big(wZ_0 + vQ \big)^2 +g(0,\bomega,\Delta,T) \ .
\end{equation}
This implies that 
\begin{eqnarray}
\label{bound_abs_moreau}
   \Big|\mathcal{M}_{g(.,\bomega,\Delta,T)}(wZ_0 + vQ,\tau)\Big|&\leq& \max\Big\{\frac{\lambda}{2\tau} \big(wZ_0 + vQ \big)^2 + \big|g(0,\bomega,\Delta,T)\big|,\big|\underset{x}{\min} \ g(x,\bomega,\Delta,T)\big| \Big\}\leq \nonumber \\
   &\leq& \frac{\lambda}{2\tau} \big(wZ_0 + vQ \big)^2 + \big|g(0,\bomega,\Delta,T)\big|+\big|\underset{x}{\min} \ g(x,\bomega,\Delta,T)\big|\ .
\end{eqnarray}
Now we show that the expectation of the right-hand side exists finite.
For the first term, we have $ \mathbb{E}\Big[\big(wZ_0 + vQ \big)^2\Big]\leq w^2 +v^2$, which is bounded by hypothesis.
For the second term, we notice that 
\begin{equation}
\label{bound_H}
    \Lambda(T|\bomega) := \sum_{k=1}^\ell\exp(\omega_k) \Psi_k(T)\leq  \sum_{k= 1}^{\ell} \exp(\omega_k)(\tau_{k+1}-\tau_{k})
\end{equation}
hence, via Cauchy Swartz inequality
\begin{eqnarray}
    \mathbb{E}\Big[ \big|g(0,\bomega,\Delta,T)\big|\Big] &=& \mathbb{E}\Big[\Big|\Delta \Lambda(T|\bomega)\Big|\Big] \leq  \sum_{k= 1}^{\ell} \exp(\omega_k)(\tau_{k+1}-\tau_{k})\leq \nonumber \\
    &\leq &\sqrt{\sum_{k= 1}^{\ell} \exp(2\omega_k) }\sqrt{\sum_{k= 1}^{\ell}(\tau_{k+1}-\tau_{k})^2}
\end{eqnarray}
which is bounded by a finite constant by assumption that $\bomega$ is in a compact set and  $\tau_{\ell+1}$ is finite.
Then,
\begin{equation}
\label{bound_logh}
    \big|\Delta \log \lambda(T|\bomega)\big| \leq \sum_{k=1}^{\ell} \psi_k(T) \big|\omega_{k}\big|
\end{equation}
which has a finite expectation value, since $ \psi_k(T)$ is an indicator function over the interval $[\tau_{k}, \tau_{k+1}]$.
Finally, we have 
\begin{equation}
     \underset{x}{\min} \ g(x,\bomega,\Delta,T) = \Delta \big( 1- \log(\Delta) + \log \Lambda(T|\bomega) \big)
\end{equation}
and we see that
\begin{equation}
\label{bound_det}
    \Big|\Delta \big( 1- \log \Delta + \log \Lambda(T|\bomega) \big)\Big|\leq 1+\Big|\log \Lambda(T|\bomega)\Big|\leq 1 + \Big|\log \sum_{k=1}^l \exp(\omega_k)(\tau_{k+1}-\tau_{k})\Big|<\infty \ .
\end{equation}
Since $\|\bomega\|$ is bounded by a large constant by assumption.
\end{proof}
\begin{proposition}[FINITE VARIANCE]
\label{fin_var}
    The random function $(\bomega,w,v,\tau) \rightarrow \mathcal{M}_{g(.,\bomega,\Delta,T)}(wZ_0 + vQ,\tau)-\Delta \log \lambda(T|\bomega)$ has a finite variance.
\end{proposition}
\begin{proof}
We compute the second moment, as we have already shown that the expectation exists finite.
Using (\ref{bound_abs_moreau}) we have
\begin{eqnarray}
   &&\Big(\mathcal{M}_{g(.,\bomega,\Delta,T)}(wZ_0 + vQ,\tau) -\Delta \log \lambda(T|\bomega)\Big)^2\leq \\
   &&\hspace{0.2cm}\Big(\frac{\lambda}{2\tau} \big(wZ_0 + vQ \big)^2 + \big|g(0,\bomega,\Delta,T)\big|+\big|\underset{x}{\min} \ g(x,\bomega,\Delta,T)\big|\Big)^2 
\end{eqnarray}
Computing the square, we see that: \\
1) $\mathbb{E}\Big[\big(wZ_0 + vQ \big)^4\Big] = w^4 \mathbb{E}\Big[Z_0^4\Big] +v^4  \mathbb{E}\Big[Q^4\Big]$ is finite, \\
2) $\mathbb{E}\Big[\big|g(0,\bomega,\Delta,T)\big|^2\Big]\leq \big( \sum_{k= 1}^{\ell} \exp(\omega_k)(\tau_{k+1}-\tau_{k})\big)^2 \leq  \big(\sum_{k= 1}^{\ell} \exp(2\omega_k)\big) \big(\sum_{k= 1}^{\ell}(\tau_{k+1}-\tau_{k})^2\big)$ which is finite, \\
3) the  term $\underset{x}{\min} \ g(x,\bomega,\Delta,T)$ is bounded by a deterministic constant (\ref{bound_det}), and hence so it is its square.
The cross terms can be similarly bounded via the Cauchy-Schwartz inequality.
\end{proof}
\begin{proposition}[FINITE VARIANCE OF THE DERIVATIVE]
\label{finite_var_der}
    The random function $(\bomega,w,v,\tau) \rightarrow \dot{g}(.,\bomega,\Delta,T)(wZ_0 + vQ,\tau)$ has a finite variance.
\end{proposition}
\begin{proof}
Using (\ref{bound_H})
\begin{eqnarray}
    \Big|\dot{g}(.,\bomega,\Delta,T)(wZ_0 + vQ,\tau)\Big| &=& \Big|\Lambda(T|\bomega)\rme^{wZ_0 + vQ} -\Delta\Big| = \nonumber \\
    &\leq& 1 +  \rme^{wZ_0 + vQ}\sqrt{\sum_{k= 1}^{\ell} \exp(2\omega_k) }\sqrt{\sum_{k= 1}^{\ell}(\tau_{k+1}-\tau_{k})^2}
\end{eqnarray}
which has a finite expectation. We then show that the second moment exists finite. Just notice that
\begin{eqnarray}
    &&\Big|\dot{g}(.,\bomega,\Delta,T)(wZ_0 + vQ,\tau)\Big|^2 \leq 1 +  2\rme^{wZ_0 + vQ}\Big( \sum_{k=1}^l\exp(2\omega_{k})\Big)^{1/2}\Big(\sum_{k=1}^{\ell}(\tau_{k+1} - \tau_k)^2\Big)^{1/2} + \nonumber\\
    && \hspace{.5cm} + \rme^{2(wZ_0 + vQ)}\sum_{k=1}^{\ell}\exp(2\omega_{k})\sum_{k=1}^{\ell}(\tau_{k+1} - \tau_k)^2
\end{eqnarray}
has a finite expectation
\end{proof}
\begin{proposition}[LIMITS OF THE MOREAU ENVELOPE]
\label{limits}
    \begin{eqnarray}
    \lim_{\tau\rightarrow \infty} \mathcal{M}_{g(.,\bomega,\Delta,T)}(x,\tau) &=& \Delta \big( 1- \log(\Delta) + \log \Lambda(T|\bomega) \big) \label{alpha_inf}\\
    \lim_{\tau\rightarrow 0 } \mathcal{M}_{g(.,\bomega,\Delta,T)}(x,\tau)  &=&   \Lambda(T|\bomega)\rme^{x} - \Delta x \label{alpha_0}\ .
    \end{eqnarray}
\end{proposition}
\begin{proof}
For (\ref{alpha_0}), observe that
\begin{equation}
    \mathcal{M}_{g(.,\bomega,\Delta,T)}\big(x,\alpha\big) := \underset{\xi}{\min}\ \Big\{\frac{1}{2\alpha}(\xi-x)^2 + g(\xi,T)\Big\} =\underset{\phi}{\min}\ \Big\{\frac{1}{2}\phi ^2 + g(x + \sqrt{\alpha} \phi ,\bomega,\Delta,T)\Big\}
\end{equation}
with $\hat{\phi}  = \underset{\phi}{\arg \min}\ \Big\{\frac{1}{2}\phi ^2 +g(x + \sqrt{\alpha} \phi ,\bomega,\Delta,T)\Big\} = - \sqrt{\alpha}\dot{g}(x + \sqrt{\alpha} \hat{\phi} ,\bomega,\Delta,T)$.
Sending $\alpha\rightarrow 0^+$ we get (\ref{alpha_0}).
For (\ref{alpha_inf}), notice that 
\begin{equation}
\frac{\partial}{\partial \alpha }\mathcal{M}_{g(.,\bomega,\Delta,T)}\big(x,\alpha\big) = - \frac{1}{2\alpha^2}\big(\prox_{g(.,\bomega,\Delta,T)}(x,\alpha) - x\big)^2 \leq 0 \ .
\end{equation}
Then 
\begin{eqnarray}
\lim_{\alpha \rightarrow \infty}\mathcal{M}_{g(.,\bomega,\Delta,T)}\big(x,\alpha\big) &=&\underset{\alpha>0}{\inf}\mathcal{M}_{g(.,\bomega,\Delta,T)}\big(x,\alpha\big)=\underset{\xi}{\min}\ \underset{\alpha>0}{\inf}\ \Big\{\frac{1}{2\alpha}(\xi-x)^2 + g(\xi,\bomega,\Delta,T)\Big\}= \nonumber \\
&=& \underset{\xi}{\min} \ g(\xi,\bomega,\Delta,T)\ .\nonumber
\end{eqnarray}
\end{proof}
\begin{proposition}[LIMITS OF THE EXPECTED MOREAU ENVELOPE]
    \begin{eqnarray}
        &&\lim_{\alpha \rightarrow \infty}\mathbb{E}_{T,Z_0,Q}\Big[\mathcal{M}_{g(.,\bomega,\Delta,T)}\big(wZ_0+vQ,\alpha\big)\Big] = \mathbb{E}_{T}\Big[ \Delta \big( 1- \log \Delta + \log \Lambda(T|\bomega) \big) \Big]\label{exp_alpha_0}\\
        &&\lim_{\alpha \rightarrow 0^+}\mathbb{E}_{T,Z_0,Q}\Big[\mathcal{M}_{g(.,\bomega,\Delta,T)}\big(wZ_0+vQ,\alpha\big)\Big] = \mathbb{E}_{T,Z_0,Q}\Big[g\big(wZ_0+vQ,T\big)\Big] \ .\label{exp_alpha_inf}
    \end{eqnarray}
\end{proposition}
\begin{proof}
For the limit of the expectation, we use the Dominated Convergence theorem. 
The Moreau envelope is integrable (proposition \ref{integrability}) and its limits exist (proposition \ref{limits}). Hence, it suffices to show that the limits are themselves integrable. First, for (\ref{exp_alpha_0}) we see that 
\begin{equation}
    \mathbb{E}\Big[\Big|\Delta \big( 1- \log \Delta + \log \Lambda(T|\bomega) \big)\Big|\Big]\leq 1+\mathbb{E}\Big[\Big|\log \Lambda(T|\bomega)\Big|\Big]\leq 1 + \Big|\log \sum_{k=1}^{\ell} \exp(\omega_{\nu})(\tau_{k+1} - \tau_k)\Big|<\infty \ .
\end{equation}
For (\ref{exp_alpha_inf}), notice that
\begin{equation}
    \mathbb{E}\Big[\Big| \Lambda(T|\bomega)\rme^{wZ_0+vQ} - \Delta (wZ_0+vQ)\Big|\Big] \leq \mathbb{E}\Big[\Big|\Lambda(T|\bomega)\rme^{wZ_0+vQ}\Big|\Big] + \mathbb{E}\Big[\Big|\Delta (wZ_0+vQ)\Big|\Big]
\end{equation}
and
\begin{eqnarray}
    &&\mathbb{E}\Big[\Big| \Lambda(T|\bomega)\rme^{wZ_0+vQ}\Big|\Big]\leq \sum_{k=1}^{\ell} \exp(\omega_{\nu}) (\tau_{k+1} - \tau_k) \rme^{2(w^2 + v^2)}<\infty\\
    &&\mathbb{E}\Big[\Big|\Delta(wZ_0+vQ)\Big|\Big] \leq \mathbb{E}\Big[\Big|wZ_0+vQ\Big|\Big]\leq \sqrt{2/\pi}\sqrt{w^2+v^2} <\infty
\end{eqnarray}
which implies the desiderata.
\end{proof}
\begin{proposition}[DERIVATIVES OF THE MOREAU ENVELOPE 1]
\label{derivatives}
    \begin{eqnarray}
        &&\frac{\partial}{\partial x} \mathcal{M}_{g(.,\bomega,\Delta,T)}(x,\tau) = \frac{1}{\tau} \big(x-\prox_{g(.,\bomega,\Delta,T)}(x,\tau)\big) = \Delta - \frac{1}{\tau}W_0\Big(\tau \Lambda(T|\bomega)\rme^{\Delta\tau +x}\Big)\\
        &&\frac{\partial}{\partial \tau} \mathcal{M}_{g(.,\bomega,\Delta,T)}(x,\tau) = -\frac{1}{2\tau^2}\big(x-\prox_{g(.,\bomega,\Delta,T)}(x,\tau)\big)^2 = \nonumber\\
        &&\hspace{1cm}-\frac{1}{2}\Big(\Delta - \frac{1}{\tau}W_0\Big(\tau \Lambda(T|\bomega)\rme^{\Delta\tau +x}\Big)\Big)^2  \ .
    \end{eqnarray}
\end{proposition}
\begin{proof}
The proposition follows from differentiation and definition of proximal mapping. 
\end{proof}
\begin{proposition}[DERIVATIVES OF THE EXPECTED MOREAU ENVELOPE 1]
    \begin{eqnarray}
    \frac{\partial}{\partial w} \mathbb{E}\Big[\mathcal{M}_{g(.,\bomega,\Delta,T)}(wZ_0+vQ,\tau) \Big]&=& \frac{1}{\tau} \mathbb{E}\Big[ Z_0 \Big(\Delta \tau - W_0\Big(\tau \Lambda(T|\bomega)\rme^{\Delta\tau +x}\Big)\Big)\Big]\\
    \frac{\partial}{\partial v} \mathbb{E}\Big[\mathcal{M}_{g(.,\bomega,\Delta,T)}(wZ_0+vQ,\tau) \Big] &=& \frac{1}{\tau} \mathbb{E}\Big[Q\Big(\Delta \tau - W_0\Big(\tau \Lambda(T|\bomega)\rme^{\Delta\tau +x}\Big)\Big)\Big]\\
    \frac{\partial}{\partial \tau}\mathbb{E}\Big[\mathcal{M}_{g(.,\bomega,\Delta,T)}(wZ_0+vQ,\tau) \Big] &=& -\frac{1}{2\tau^2}\mathbb{E}\Big[\Big(\Delta \tau - W_0\Big(\tau \Lambda(T|\bomega)\rme^{\Delta\tau +x}\Big)\Big)^2\Big] \ .
\end{eqnarray}
\end{proposition}
\begin{proof}
    We show that the conditions of the dominated converge theorem of Lebesgue are satisfied simultaneously for all derivatives.
    Via proposition \ref{derivatives} we see that
    \begin{eqnarray}
        \mathbb{E}\Big[\Big|\frac{\partial}{\partial w} \mathcal{M}_{g(.,T)}(wZ_0+vQ,\tau)\Big| \Big]&\leq&  \frac{1}{\tau} \mathbb{E}\Big[Z_0^2\Big]^{1/2}\mathbb{E}\Big[(\Delta \tau - W_0\Big(\tau \Lambda(T|\bomega)\rme^{\Delta\tau +wZ_0+vQ}\Big)^2\Big]^{1/2} \\
        \mathbb{E}\Big[\Big|\frac{\partial}{\partial v} \mathcal{M}_{g(.,T)}(wZ_0+vQ,\tau)\Big| \Big]&\leq&  \frac{1}{\tau} \mathbb{E}\Big[Q^2\Big]^{1/2}\mathbb{E}\Big[(\Delta \tau - W_0\Big(\tau \Lambda(T|\bomega)\rme^{\Delta\tau +wZ_0+vQ}\Big)^2\Big]^{1/2} \\
        \mathbb{E}\Big[\Big|\frac{\partial}{\partial \tau} \mathcal{M}_{g(.,T)}(wZ_0+vQ,\tau)\Big| \Big]&=&  \frac{1}{2\tau^2}  \mathbb{E}\Big[(\Delta \tau - W_0\Big(\tau \Lambda(T|\bomega)\rme^{\Delta\tau +wZ_0+vQ}\Big)^2\Big] \ .
    \end{eqnarray}
    Thus it is sufficient to show that the last term is bounded to show integrability of all the first derivatives.\\
    Since $\tau>0$ and for $x\geq0$ it holds $W_0(x)\geq0$, we have that 
    \begin{equation}
        \bigg(\Delta\tau - W_0\Big(\tau \Lambda(T|\bomega)\rme^{\Delta\tau +wZ_0+vQ}\Big)\bigg)^2\leq \tau^2 + W_0^2\Big(\tau \Lambda(T|\bomega)\rme^{\Delta\tau +wZ_0+vQ}\Big) \ .
    \end{equation}
    We now use the bound $W_0(x)\leq \log (x+1)\leq x  $ which is valid for $x\geq 0$, obtaining
    \begin{eqnarray}
         \frac{1}{\tau^2}\mathbb{E}\Bigg[\bigg(\Delta\tau - W_0\Big(\tau \Lambda(T)\rme^{\Delta\tau +wZ_0+vQ}\Big)\bigg)^2\Bigg]&\leq& 1 + 1 \mathbb{E}\Big[ \Lambda^2(T|\bomega)\rme^{2(\Delta\tau +wZ_0+vQ)}\Big]\ .
    \end{eqnarray}
    which can be shown to be finite using (\ref{bound_H}), since $\|\bomega\|,w,v$ are bounded by a constant, $\tau>0$ and $\tau_\ell$ is finite by hypothesis.
\end{proof}
\begin{proposition}[PARTIAL DERIVATIVES OF THE EXPECTED MOREAU ENVELOPE 2]
    \begin{eqnarray}
    \frac{\partial}{\partial \omega_k} \mathbb{E}\Bigg[\mathcal{M}_{g(.,\bomega,\Delta,T)}(wZ_0+vQ,\tau) \Bigg]&=&  \mathbb{E}\Big[ \Psi_k(T)\rme^{\prox_{g(.,\bomega,\Delta,T)}(wZ_0+vQ,\tau)}\exp(\omega_k)\Big] \\
    \frac{\partial}{\partial \omega_k} \mathbb{E}\Big[\Delta \log \lambda(T|\bomega)\Big]&=& \mathbb{E}\Big[ \Delta \psi_k(T)\Big]
\end{eqnarray}
\end{proposition}
\begin{proof}
    We show that the conditions of the dominated convergence theorem are satisfied.
    First we compute 
    \begin{eqnarray}
        \frac{\partial}{\partial \omega_k} \mathcal{M}_{g(.,\bomega,\Delta,T)}(wZ_0+vQ,\tau) &=& \Psi_k(T)\exp(\omega_k)\frac{W_0\Big(\tau \Lambda(T|\bomega)\rme^{\Delta\tau +wZ_0+vQ}\Big)}{\tau \Lambda(T|\bomega)}\nonumber \\
        \frac{\partial}{\partial \omega_k} \Delta \log \lambda(T|\bomega)&=& \Delta \psi_k(T)
    \end{eqnarray}
    where we used that
    \begin{eqnarray}
        \exp\Big\{\prox_{g(.,T)}(wZ_0+vQ,\tau)\Big\} &=&\exp\Big\{wZ_0 + v Q +\Delta \tau - W_0\Big(\tau \Lambda(T|\bomega)\rme^{\Delta\tau+wZ_0 + v Q}\Big)\Big\} = \nonumber\\
        &=&\frac{W_0\Big(\tau \Lambda(T|\bomega)\rme^{\Delta\tau +x}\Big)}{\tau \Lambda(T|\bomega)} \ .\nonumber
    \end{eqnarray}
    Notice that
    \begin{eqnarray}
        &&\mathbb{E}\Bigg[\Psi_k(T)\exp(\omega_k)\frac{W_0\Big(\tau \Lambda(T|\bomega)\rme^{\Delta\tau +wZ_0+vQ}\Big)}{\tau \Lambda(T|\bomega)}\Bigg]\leq (\tau_{k+1}-\tau_k)\mathbb{E}_{Z_0,Q}\Bigg[\frac{W_0\Big(\tau \Lambda(T|\bomega)\rme^{\Delta\tau +wZ_0+vQ}\Big)}{\tau \Lambda(T|\bomega)} \Bigg] \nonumber \\
        &&\hspace{1cm}\leq (\tau_{k+1}-\tau_k)\mathbb{E}_{Z_0,Q}\Big[\rme^{\Delta\tau +wZ_0+vQ}\Big], 
    \end{eqnarray}
    where  for the last inequality we used that $W_0(x)\leq \log(x+1)\leq x$, for $x\geq0$. This implies that the above exists finite.
    Furthermore
    \begin{equation}
       \mathbb{E}\Big[ \Delta \psi_k(T)\Big]  \leq \mathbb{E}\Big[\psi_k(T)\Big] 
    \end{equation}\
    which is finite since $\psi_k$ is the indicator function of the interval $(\tau_k, \tau_{k+1})$.
    
\end{proof}

\section{POINTWISE CONVERGENCE OF $\mathcal{L}_n$ TO $\mathcal{L}$}
\label{appendix:pointwise}
\begin{proposition}
Let 
\begin{eqnarray}
   \mathcal{L}_n(\bomega,w,v,\phi,\tau) &:=& \frac{1}{n} \sum_{i=1}^n \mathcal{M}_{g(.,\bomega,\Delta_i,T_i)  }\big(wZ_{0,i} +  vQ_i,\tau/\phi\big) - \Delta_i \log \lambda(T_i|\bomega)+\nonumber \\
   &+& \phi\big(\tau/2- v\|\mathbf{G}\|/\sqrt{n}\big)+ \frac{1}{2}\eta (v^2 + w^2) +\frac{1}{2}\alpha \|\bomega\|^2
\end{eqnarray}
and 
\begin{eqnarray}
    \mathcal{L}(\bomega,w,v,\phi,\tau) &:=& \mathbb{E}_{T,Z_0,Q}\Big[\mathcal{M}_{g(.,\bomega,\Delta,T)  }\big(wZ_{0}  +v Q,\tau/\phi\big)\Big] - \mathbb{E}\Big[ \Delta \log \lambda(T|\bomega)\Big] +\nonumber \\
    &+&  \phi\big(\tau/2- v\sqrt{\zeta}\big)+ \frac{1}{2}\eta (v^2 + w^2) +\frac{1}{2}\alpha\|\bomega\|^2\ .
\end{eqnarray}
Then $\mathcal{L}_n(w,v,\bomega, \phi,\tau)  \xrightarrow[n\rightarrow \infty]{P} \mathcal{L}(w,v,\bomega,\phi,\tau)$ pointwise in $w,v,\tau,\bomega, \phi$.
   
\end{proposition}
\begin{proof}
Since $\|\mathbf{G}\|$ follows a chi distribution with $p-1$ degrees of freedom
\begin{equation}
    \mathbb{E}\Big[\|\mathbf{G}\|\Big] = \sqrt{2} \frac{\Gamma(\frac{p}{2})}{\Gamma(\frac{p-1}{2})} = \sqrt{p-2} + o(1),
    \qquad \mathbb{V}\Big[\|\mathbf{G}\|^2\Big] = p-1 - \mathbb{E}\Big[\|\mathbf{G}\|\Big]^2 = 1 + o(1),  
\end{equation}
hence we conclude $\|\mathbf{G}\|/\sqrt{n} = \sqrt{p/n} + o_P(1)\xrightarrow[n\rightarrow \infty]{P} \sqrt{\zeta}$, where $\zeta := \lim_{n\rightarrow \infty} p(n)/n$ (e.g. via Chebyshev inequality).
Furthermore, by the weak law of large numbers, given proposition \ref{fin_var}, we have that 
\begin{eqnarray}
    \hspace{-1cm}&&\frac{1}{n} \sum_{i=1}^n\mathcal{M}_{g(.,\bomega,\Delta_i,T_i)  }\big(wZ_{0,i} +  vQ_i,\tau/\phi\big)- \Delta_i \log \lambda(T_i|\bomega)\xrightarrow[n\rightarrow \infty]{P}\\
    \hspace{-1cm}&&\hspace{1cm}\mathbb{E}\Big[\mathcal{M}_{g(.,\bomega,\Delta,T)  }\big(wZ_{0} +  vQ,\tau/\phi\big)-\Delta \log \lambda(T|\bomega)\Big] \ .
\end{eqnarray}
\end{proof}

\section{ASYMPTOTIC CONVERGENCE OF THE SADDLE POINT OF THE AO}
\label{appendix:saddle_convergence}
In this section we show that
\begin{equation}
     \underset{\bomega, w,v\geq 0}{\min} \ \underset{0\leq\phi\leq \lambda_{\max} }{\max} \ \underset{\tau> 0}{\min} \  \mathcal{L}_n(\bomega, w,v,\phi,\tau)  \xrightarrow[n,p\rightarrow \infty]{P}  \underset{\bomega, w,v\geq 0}{\min} \ \underset{0\leq\phi\leq \lambda_{\max} }{\max} \ \underset{\tau> 0}{\min}  \ \mathcal{L}(\bomega, w,v,\phi,\tau)
\end{equation}
The idea is to use the so-called \emph{convexity lemma}, this guarantees that point-wise convergence of convex functions over compact sets implies uniform convergence over the latter. 
\begin{lemma}[CONVEXITY LEMMA]
\label{lemma:convexity}
    Let $E\subset \mathbb{R}^d$ be open, convex and $F_n$ a sequence of proper, convex functions and $f$ a deterministic function, both defined on $E$, such that 
    \begin{equation}
        F_n(x) \xrightarrow[]{P}f(x), \forall \ x \in E \ .
    \end{equation}
    Then $F_n$ converges uniformly to $f$ over all compact subsets of $E$, i.e. 
    \begin{equation}
        \underset{x\in A}{\sup} \Big|F_n(x)-f(x)\Big| \xrightarrow[]{P} 0 \ .
    \end{equation}
\end{lemma}
The proof can be found in \cite{Gill_82} (page 1116 theorem 2.1).
In particular, we will use the following lemma, which is a consequence of the former.
\begin{lemma}[MIN-CONVERGENCE - OPEN SETS]
\label{lemma:min_conv}
     Consider a sequence of proper, convex stochastic functions $F_n : (0, \infty) \rightarrow \mathbb{R}$, and, a deterministic function $f : (0, \infty) \rightarrow \mathbb{R}$, such that:
     \begin{itemize}
         \item[(a)] $F_n(x) \xrightarrow[]{P}f(x)$, $\forall \ x > 0 $,
         \item[(b)] $\exists z>0\ : \ f(x)>\underset{x>0}{\inf} f(x), \ \forall x \geq z$ $\iff$ $\lim_{x\rightarrow\infty} f(x) = +\infty$.
     \end{itemize}
     Then 
     \begin{equation}
         \underset{x>0}{\inf} F_n(x) \xrightarrow[]{P} \underset{x>0}{\inf}  f(x) \ .
     \end{equation}
\end{lemma}
The proof can be found in \cite{Thrampoulidis_2018} (lemma A6 page 29).
The assumption $(b)$ of lemma \ref{lemma:min_conv} above, is known as level-bounded condition. The following lemma, gives an equivalent characterization of a level bounded convex function.
\begin{lemma}[LEVEL BOUNDED CONVEX FUNCTIONS]
\label{lemma:level_bounded}
     Let $f : (0, \infty) \rightarrow \mathbb{R}$be convex, then the following statements are equivalent:
     \begin{itemize}
         \item[(a)] $\exists z>0\ : \ f(x)>\underset{x>0}{\inf} f(x), \ \forall x \geq z$,
         \item[(b)] $\lim_{x\rightarrow\infty} f(x) = +\infty$.
     \end{itemize}
\end{lemma}
The proof can be found in \cite{Thrampoulidis_2018} (lemma A7 page 29).
We will proceed sequentially from the innermost operation to the outermost.

\subsection{Minimization over $\tau$}
Fix $\bomega,w,v,\phi\geq 0$, and let us denote $\mathcal{L}_n^{\bomega,w,v,\phi}(\tau):= \mathcal{L}_n(\bomega,w,v,\phi,\tau)$ and  $\mathcal{L}^{\bomega,w,v,\phi}(\tau):= \mathcal{L}(\bomega,w,v,\phi,\tau)$. 
\begin{proposition}
    \begin{equation}
    \underset{\tau\geq 0 }{\inf} \ \mathcal{L}_n^{\bomega,w,v,\phi}(\tau)\xrightarrow[]{P} \underset{\tau>0}{\inf} \ \mathcal{L}^{\bomega,w,v,\phi}(\tau) \ .
\end{equation}
\end{proposition}

    \begin{proof}     
    Since $\mathcal{L}_n^{\bomega,w,v,\phi}(.)$ is convex, and $\mathcal{L}_n^{\bomega,w,v,\phi}(.)\xrightarrow[]{P} \mathcal{L}^{\bomega,w,v,\phi}(.)$ on $(0,\infty)$, also $\mathcal{L}^{\bomega,w,v,\phi}(.)$ is convex. 
    If $\phi>0$, then $ \lim_{\tau \rightarrow\infty}\mathcal{L}^{\bomega,w,v,\phi}(\tau) = +\infty$, since $\phi\tau \rightarrow \infty$ ($\phi>0$). Hence, via lemma \ref{lemma:level_bounded}, we see that the conditions of lemma \ref{lemma:min_conv} are satisfied and the conclusion follows.
    If $\phi = 0$, then 
    \begin{eqnarray}
        \underset{\tau\geq 0 }{\inf} \ \mathcal{L}_n^{\bomega,w,v,\phi=0}(\tau) &=& \lim_{\alpha \rightarrow \infty} \frac{1}{n} \sum_{i=1}^n\mathcal{M}_{g (.,\bomega,\Delta,T)}\big(wZ_{0,i}+vQ_i,\alpha\big) - \Delta_i\log \lambda(T_i|\bomega)+\frac{1}{2}\alpha\|\bomega\|^2=\nonumber \\
        &=&\frac{1}{n} \sum_{i=1}^n \underset{\xi}{\min} \ g(\xi,\bomega,\Delta_i,T_i)- \Delta_i\log \lambda(T_i|\bomega) +\frac{1}{2}\alpha\|\bomega\|^2\ .
    \end{eqnarray}
    Under our assumptions also 
    \begin{eqnarray}
        \underset{\tau\geq 0 }{\inf} \ \mathcal{L}^{\bomega,w,v,\phi=0}(\tau) &=& \lim_{\alpha \rightarrow \infty}\mathbb{E}_{T,Z_0,Q}\Big[\mathcal{M}_{g (.,\bomega,\Delta,T)}\big(wZ_0+vQ,\alpha\big)- \Delta\log \lambda(T|\bomega)\Big] +\frac{1}{2}\alpha\|\bomega\|^2=\nonumber\\
        &=&\mathbb{E}_{T}\Big[ \underset{\xi}{\min} \ g(\xi,\bomega,\Delta,T)- \Delta\log \lambda(T|\bomega)\Big] +\frac{1}{2}\alpha\|\bomega\|^2
    \end{eqnarray}
    and via the weak law of large numbers, we obtain the claim.
\end{proof}

\subsection{Maximization over $\phi$}
Fix $\bomega,w,v$, let us denote $\mathcal{L}_n^{\bomega,w,v}(\phi):= \underset{\tau>0}{\inf}\ \mathcal{L}_n(\bomega,w,v,\phi,\tau)$ and  $\mathcal{L}^{\bomega,w,v}(\phi):= \underset{\tau>0}{\inf}\ \mathcal{L}(\bomega,w,v,\phi,\tau)$. 
\begin{proposition}
    \begin{equation}
    \underset{\phi\geq 0}{\sup} \ \mathcal{L}_n^{\bomega,w,v}(\phi)\xrightarrow[]{P} \underset{\phi\geq 0}{\sup} \ \mathcal{L}^{\bomega,w,v}(\phi) \ .
    \end{equation}
\end{proposition}
\begin{proof}
Notice that $\mathcal{L}_n^{\bomega,w,v}(\phi)$ is concave, as the pointwise minima of concave functions. The same is true for $\mathcal{L}^{\bomega,w,v}(\phi)$ because $\mathcal{L}_n^{\bomega,w,v}(.)\xrightarrow[]{P}\mathcal{L}_n^{\bomega,w,v}(.)$ pointwise.

If $\phi=0$, then $\mathcal{L}_n^{\bomega,w,v>0}(0)\xrightarrow[]{P}\ \mathcal{L}^{\bomega,w,v>0}(0)$ for any $v$, as shown before.
Else $\phi >0$. We want to use again lemma \ref{lemma:min_conv}: if we show that $\lim_{\phi\rightarrow \infty}\mathcal{L}_n^{\bomega,w,v}(\phi) = - \infty$, then we are done, as the function is concave and level bounded.
It holds by definition that 
\begin{equation}
   \mathcal{L}^{\bomega,w,v}(\phi) \leq \lim_{\tau \rightarrow 0^+}\mathcal{L}(\bomega,w,v,\phi,\tau) 
\end{equation}
where
\begin{eqnarray}
    \lim_{\tau \rightarrow 0^+}\mathcal{L}(\bomega,w,v,\phi ,\tau)  &=&   \lim_{\tau \rightarrow 0^+}\mathbb{E}\Big[\mathcal{M}_{ g (.,\bomega,\Delta,T)}\big(wZ_{0}  +v Q,\tau/\phi\big)-\Delta\log \lambda(T|\bomega)\Big] -\phi v\sqrt{\zeta} +\frac{1}{2}\alpha\|\bomega\|^2=\nonumber \\
    &=&\mathbb{E}\Big[g\big(wZ_{0}  +v Q,\bomega,\Delta,T\big)-\Delta\log \lambda(T|\bomega)\Big] -\phi v\sqrt{\zeta} +\frac{1}{2}\alpha\|\bomega\|^2\ .
\end{eqnarray}
If  $v>0$, then $\lim_{\phi\rightarrow\infty} \mathcal{L}^{\bomega,w,v}(\phi) = -\infty $, since the expectation of $g$ exists finite for fixed $\bomega,w,v$, and lemma \ref{lemma:min_conv} gives 
\begin{equation}
    \underset{\phi\geq0}{\sup} \ \mathcal{L}_n^{\bomega,w,v}(\phi)\xrightarrow[]{P} \underset{\phi\geq 0}{\sup} \ \mathcal{L}^{\bomega,w,v}(\phi) \ .
\end{equation}
On the other hand if $v =0$, it is not straightforward to conclude the desiderata.
Observe that 
\begin{equation}
    \lim_{\phi\rightarrow \infty}\ \underset{\tau>0}{\inf}\mathcal{L}(\bomega,w,0,\phi ,\tau)  \leq \underset{\phi\geq0}{\sup} \ \underset{\tau>0}{\inf}\mathcal{L}(\bomega,w,0,\phi ,\tau)  \leq\underset{\phi\geq0}{\sup} \lim_{\tau \rightarrow 0^+}\mathcal{L}(\bomega,w,0,\phi ,\tau) 
\end{equation}
and 
\begin{equation}
    \lim_{\phi\rightarrow \infty}\ \underset{\tau>0}{\inf}\mathcal{L}_n(\bomega,w,0,\phi ,\tau)  \leq \underset{\phi\geq0}{\sup} \ \underset{\tau>0}{\inf}\mathcal{L}_n(\bomega,w,0,\phi ,\tau)  \leq \underset{\phi\geq0}{\sup}\lim_{\tau \rightarrow 0^+}\mathcal{L}_n(\bomega,w,0,\phi ,\tau)  \ .
\end{equation}
We now show that 
\begin{eqnarray}
    &&\lim_{\phi\rightarrow \infty}\ \underset{\tau>0}{\inf}\mathcal{L}_n(\bomega,w,0,\phi ,\tau)\xrightarrow[]{P}\lim_{\phi\rightarrow \infty}\ \underset{\tau>0}{\inf}\mathcal{L}(\bomega,w,0,\phi ,\tau)\\
    &&\lim_{\tau \rightarrow 0^+}\mathcal{L}_n(\bomega,w,0,\phi ,\tau) \xrightarrow[]{P} \lim_{\tau \rightarrow 0^+}\mathcal{L}(\bomega,w,0,\phi ,\tau)  \ .
\end{eqnarray}
Which implies
\begin{equation}
     \underset{\phi\geq0}{\sup} \ \underset{\tau>0}{\inf}\mathcal{L}_n(\bomega,w,0,\phi ,\tau) \xrightarrow[]{P}\underset{\phi\geq0}{\sup} \ \underset{\tau>0}{\inf}\mathcal{L} (\bomega,w,0,\phi ,\tau)
\end{equation}
and as a consequence the desired.

By the weak law of large numbers, we have that
\begin{eqnarray}
    &&\lim_{\tau \rightarrow 0^+}\mathcal{L}_n(\bomega,w,v=0,\phi ,\tau) = \frac{1}{n}\sum_{i=1}^n g\big(wZ_{0},\bomega,\Delta,T\big) - \Delta_i\log \lambda(T_i|\bomega)+\frac{1}{2}\alpha\|\bomega\|^2\nonumber \\
    &&\hspace{1cm}\xrightarrow[]{P} \mathbb{E}_{\Delta,T,Z_0,Q}\Big[g\big(wZ_{0},\bomega,\Delta,T\big)-\Delta\log \lambda(T|\bomega)\Big]  +\frac{1}{2}\alpha\|\bomega\|^2=  \lim_{\tau \rightarrow 0^+}\mathcal{L}(\bomega,w,v=0,\phi ,\tau)\ .
\end{eqnarray}
The \say{other side} requires more work. Notice that 
\begin{eqnarray}
    \lim_{\phi\rightarrow \infty}\ \underset{\tau>0}{\inf}\mathcal{L}_n(\bomega,w,v=0,\phi,\tau) &=& \lim_{\phi\rightarrow \infty}\underset{\kappa>0}{\inf}\frac{1}{n}\sum_{i=1}^n \mathcal{M}_{ g (.,\bomega,\Delta_i,T_i)}\big(wZ_{0,i},\kappa\big) -\Delta_i\log \lambda(T_i|\bomega) +  \nonumber\\
    &+& \phi^2\kappa /2  +\frac{1}{2}\alpha\|\bomega\|^2
\end{eqnarray}
The sequence of functions $f_n^{\bomega,w,\phi}(\kappa):= \frac{1}{n}\sum_{i=1}^n \mathcal{M}_{ g (.,\bomega,\Delta_i,T_i)}\big(wZ_{0,i},\kappa\big)-\Delta_i\log \lambda(T_i|\bomega) + \phi^2\kappa /2$ : 
\begin{itemize}
    \item are convex in $\kappa$, 
    \item  $f_n'(0^+)>0$, for $\phi,n$ sufficiently large, since $\mathbb{V}\Big[ \mathcal{M}_{ g (.,\bomega,\Delta,T)}\big(wZ_{0},\kappa\big)-\Delta\log \lambda(T|\bomega) \Big]<\infty$
    \item $\lim_{\kappa\rightarrow \infty}f^{\bomega,w,\phi}_n(\kappa) = \infty$ \ .
\end{itemize}
Because of convexity of $f^{\bomega,w,\phi}_n$, these imply that the infimum is at $\kappa = 0$, for sufficiently large $n,\phi$. Hence,
\begin{equation}
         \lim_{\phi\rightarrow \infty}\ \underset{\tau>0}{\inf}\mathcal{L}_n(\bomega,w,0,\phi ,\tau)  \xrightarrow[]{P} \mathbb{E}_{T,Z_0,Q}\Big[g\big(wZ_{0},\bomega,\Delta,T\big)-\Delta\log \lambda(T|\bomega)\Big] + \frac{1}{2}\alpha\|\bomega\|^2\ .
\end{equation}
Similarly
\begin{equation}
    \lim_{\phi\rightarrow \infty}\underset{\tau>0}{\inf}\mathcal{L}(\bomega,w,0,\phi ,\tau)  =  
     \lim_{\phi\rightarrow \infty}\underset{\kappa>0}{\inf}\mathbb{E}_{T,Z_0,Q}\Big[\mathcal{M}_{ g (.,\bomega,\Delta,T)}\big(wZ_{0},\kappa\big)-\Delta \log \lambda(T|\bomega)\Big] +\phi^2\kappa /2 + \frac{1}{2}\alpha\|\bomega\|^2   \ .
\end{equation}
The function $f^{\bomega,w,\phi}(\kappa) := \mathbb{E}_{T,Z_0,Q}\Big[\mathcal{M}_{ g (.,\bomega,\Delta,T)}\big(wZ_{0},\kappa\big)-\Delta\log \lambda(T|\bomega)\Big] +\phi^2\kappa /2$ : 
\begin{itemize}
    \item is convex in $\kappa$, 
    \item  $f'(0^+)>0$, for $\phi$ sufficiently large,
    \item $\lim_{\kappa\rightarrow \infty}f^{\bomega,w,\phi}(\kappa) = \infty$ \ .
\end{itemize}
Because of convexity of $f^{\bomega,w,\phi}$, these imply that the infimum is at $\kappa = 0$. Hence
\begin{equation}
    \lim_{\phi\rightarrow \infty}\underset{\kappa>0}{\inf}\mathcal{L}(\bomega,w,0,\phi ,\tau) = \mathbb{E}_{T,Z_0,Q}\Big[g\big(wZ_{0},\bomega,\Delta,T\big)-\Delta\log \lambda(T|\bomega)\Big]+ \frac{1}{2}\alpha\|\bomega\|^2
\end{equation}
and we have shown
\begin{eqnarray}
     &&\lim_{\phi\rightarrow \infty}\ \underset{\tau>0}{\inf}\mathcal{L}_n(\bomega,w,0,\phi ,\tau) = \frac{1}{n}\sum_{i=1}^n g\big(wZ_{0},\bomega,\Delta,T\big) \nonumber \\
     &&\hspace{1cm}\xrightarrow[]{P} \mathbb{E}_{T,Z_0,Q}\Big[g\big(wZ_{0},\bomega,\Delta,T\big)-\Delta\log \lambda(T|\bomega)\Big]+ \frac{1}{2}\alpha\|\bomega\|^2 = \lim_{\phi\rightarrow \infty}\underset{\kappa>0}{\inf}\mathcal{L}(\bomega,w,0,\phi ,\tau) \ .\nonumber 
\end{eqnarray}
\end{proof}

\subsection{Minimization over $\bomega, w, v$}
Let us denote $\mathcal{L}_n(\bomega,w,v):= \underset{\phi\geq 0}{\sup} \ \mathcal{L}_n^{\bomega,w,v>0}(\phi)$ and $\mathcal{L}(\bomega,w,v):=  \underset{\phi\geq 0}{\sup} \ \mathcal{L}^{\bomega,w,v}(\phi)$. 
\begin{proposition}
    The function $\mathcal{L}(\bomega,w,v)$ admits a unique minimizer for $\bomega \in \mathcal{S}_{\bomega}$, $w\in[0,C_{\bbeta}]$ and $v\in [0,C_{\bbeta}]$ and 
    \begin{equation}
    \underset{(\bomega,w,v)}{\min}\mathcal{L}_n(\bomega,w,v) \xrightarrow[n\rightarrow \infty]{P} \underset{(\bomega,w,v)}{\min}\mathcal{L}(\bomega,w,v) \ .
    \end{equation}
\end{proposition}
\begin{proof}
The functions $\mathcal{L}_n$ are convex in their arguments, as the pointwise maxima  of convex functions. The same is true for $\mathcal{L}$, furthermore,  $\mathcal{L}_n\xrightarrow[]{P}\mathcal{L}$ pointwise. Let us denote $\mathcal{S}_{\bomega,w,v} =  \mathcal{S}_{\bomega}\times [0,C_{\bbeta}]\times [0,C_{\bbeta}]$. 
The function $\mathcal{L}(\bomega,w,v)$ is strongly convex in $w,v$ and strictly convex in $\bomega$ (it is sufficient to compute the second derivative, which is always non-negative on $\mathcal{S}_{\bomega}$). In particular, $\mathcal{L}(\bomega,w,v)$ is strongly convex in its arguments if we insist that $\forall \mu, \ P[t_{\mu}<T<t_{\mu+1}]\geq\delta>0$ and hence admits a unique minimizer.
If $\bomega_{\star},w_{\star},v_{\star}$ is in the interior of $\mathcal{S}_{\bomega,w,v}$, then we can readily conclude via the convexity lemma\ref{lemma:convexity}, that
\begin{equation}
    \underset{(\bomega,w,v) }{\min}\mathcal{L}_n(\bomega,w,v) \xrightarrow[n\rightarrow\infty]{P} \underset{(\bomega,w,v)}{\min}\mathcal{L}(\bomega,w,v) \ .
\end{equation}
\end{proof}

\section{CONVERGENCE IN PROBABILITY OF THE MINIMIZER}
\label{appendix:conv_minimizer}
In order to prove (\ref{w_n},\ref{v_n},\ref{lambda_n}), we want to use Theorem \ref{CGMT}, and, in particular, implication (\ref{asymptotic_cgmt}). To do so, let us define 
\begin{equation}
    \mathcal{S}_{\epsilon} := \big\{ (\bbeta,\bomega) \in \mathcal{S}_{\bbeta}\times \mathcal{S}_{\bomega}\\ : \ \big|w_n-w_{\star}\big|<\epsilon, \big|v_n-v_{\star}\big|<\epsilon, \big|\bomega-\bomega_{\star}\big|<\epsilon\big\}
\end{equation}
with $\mathcal{S}_{\bbeta},\mathcal{S}_{\bomega}$ two compact subsets of $\mathbb{R}^{p}$ and $\mathbb{R}^d$, respectively,  and $w_n,v_n$ defined as 
\begin{equation}
        w_n := \frac{\bbeta_0'\bbeta}{\|\bbeta_0\|}, \quad v_n :=\big\|\mathbf{P}_{\perp\bbeta_0}\bbeta\big\|= \Big\|\Big(\bm{I} - \frac{\bbeta_0\bbeta_0'}{\|\bbeta_0\|^2}\Big)\bbeta\Big\|  \ .
\end{equation}
By the convexity lemma we have uniform convergence $\mathcal{L}_n(\bomega,w,v)\xrightarrow[n\rightarrow\infty]{P}\mathcal{L}(\bomega,w,v)$ for $(\bomega,w,v) \in \mathcal{B}\subset  \mathcal{S}_{\bbeta}\times \mathcal{S}_{\bomega}$, with $\mathcal{B}$ compact. As a consequence, since $\mathcal{S}_{\epsilon}^c:= \mathcal{S}_{\bbeta}\times \mathcal{S}_{\bomega}\setminus \mathcal{S}_{\epsilon}$ is compact, we have that
\begin{equation}
    \underset{(\bomega,w,v) \in \mathcal{S}_{\epsilon}^c}{\min}\mathcal{L}_n(\bomega,w,v) \xrightarrow[n\rightarrow\infty]{P} \underset{(\bomega,w,v) \in \mathcal{S}_{\epsilon}^c}{\min}\mathcal{L}(\bomega,w,v) \ .
\end{equation}
The function $\mathcal{L}$ above is strongly convex in its arguments, hence it  has a unique minimizer $(w_{\star},v_{\star},\bomega_{\star})$ as defined in (\ref{w_n},\ref{v_n},\ref{lambda_n}) respectively.
In this case it must hold that
\begin{equation}
    \underset{(\bomega,w,v) \in \mathcal{S}_{\epsilon}^c}{\min}\mathcal{L}(\bomega,w,v) > \underset{(\bomega,w,v)}{\min}\mathcal{L}(\bomega,w,v)  \ .
\end{equation}
Via Theorem \ref{CGMT} implication (\ref{asymptotic_cgmt}) this implies that 
\begin{equation}
    \lim_{n\rightarrow\infty} P\Big[(\hat{\bbeta}_n,\hat{\bomega}_n)\in  \mathcal{S}_{\epsilon}\Big] = 1 
\end{equation}
which is the desired.

\end{document}